\documentclass{amsart}

\usepackage{mathtools}
\mathtoolsset{showonlyrefs}
\usepackage{amssymb}
\usepackage{amsthm}
\usepackage{physics2}
\usephysicsmodule{ab, ab.braket}
\usepackage{mleftright}
\mleftright
\usepackage[shortlabels]{enumitem}
\usepackage[normal, nopartial]{fixdif}
\usepackage[nobreak]{cite}
\usepackage{zref-clever}
\zcsetup{cap}
\usepackage{hyperref}

\DeclareMathOperator{\Mat}{Mat}
\DeclareMathOperator{\supp}{supp}
\NewDocumentCommand{\Z}{}{\mathbb{Z}}
\NewDocumentCommand{\R}{}{\mathbb{R}}
\NewDocumentCommand{\C}{}{\mathbb{C}}
\NewDocumentCommand{\abs}{m}{\left\lvert #1 \right\rvert}
\NewDocumentCommand{\norm}{m}{\left\lVert #1 \right\rVert}

\ExplSyntaxOn

\NewDocumentCommand{\set}{m}{
    \tl_set:Nn \l_tmpa_tl {#1}
    \tl_replace_once:Nnn \l_tmpa_tl {|} {\mathrel{}\middle\vert\mathrel{}}
    \left\{ \l_tmpa_tl \right\}
}

\hook_gput_code:nnn {env / lemma / begin} {.} {
    \zcsetup{countertype = {theorem = lemma}}
}
\hook_gput_code:nnn {env / proposition / begin} {.} {
    \zcsetup{countertype = {theorem = proposition}}
}

\ExplSyntaxOff

\theoremstyle{plain}
\newtheorem{theorem}{Theorem}[section]
\newtheorem{proposition}[theorem]{Proposition}
\newtheorem{lemma}[theorem]{Lemma}
\newtheorem{fact}[theorem]{Fact}

\theoremstyle{definition}
\newtheorem{definition}[theorem]{Definition}
\newtheorem{example}[theorem]{Example}

\title{On a Compact Generalization of Association Schemes}

\author{Akifumi Nakada}
\address{
    Mathematics Program,
    Graduate School of Advanced Science and Engineering,
    Hiroshima University,
    Higashi-Hiroshima 739-8526 JAPAN
}
\email{nakada-aki@hiroshima-u.ac.jp}

\date{\today}
\thanks{The author is supported by Grant-in-aid for Science Research, No. 23KJ1641}

\begin{document}

\begin{abstract}
    We introduce a notion of \emph{compact association schemes}, which serves as a compact analogue of classical (finite) association schemes.
    Our definition is formulated in a way that closely parallels the finite case, naturally admits a Bose--Mesner algebra, and includes the compact \emph{strong} continuous association schemes introduced by Voit [J.~Aust.~Math.~(2019)] within the framework of hypergroups.
    This approach provides a new perspective that bridges the theory of association schemes with harmonic analysis on compact homogeneous spaces.
\end{abstract}

\maketitle

\section{Introduction}

Harmonic analysis on homogeneous spaces of finite groups has been extensively developed through representation theory.
A remarkable combinatorial abstraction of such spaces is the notion of an association scheme, defined on a finite set by specifying a family of binary relations satisfying certain axioms:

\begin{definition}
    Let $X$ and $\mathcal{I}$ be finite sets, and let $R \colon X \times X \to \mathcal{I}$ be a surjective map.
    We call $R \coloneq (X,R)$ an \emph{association scheme} (AS) if it satisfies the following conditions:
    \begin{enumerate}[({AS}1), series = AS]
        \item There exists $i_0 \in \mathcal{I}$ such that $R^{-1}(i_0) = \Delta_X \coloneq \set{(x,x) | x \in X}$.
        \item For all $i,i',k \in \mathcal{I}$, there exists $p_{i,i'}^k \in \Z_{\ge 0}$ such that for all $(x,z) \in R^{-1}(k)$,
        \begin{equation}
            p_{i,i'}^k
            = \# \set{y \in X | R(x,y) = i, R(y,z) = i'}.
        \end{equation}
        \item For all $i \in \mathcal{I}$, there exists $i^\top \in \mathcal{I}$ such that
        \begin{equation}
            R^{-1}(i^\top)
            = R^{-1}(i)^\top
            \coloneq \set{(y,x) | (x,y) \in R^{-1}(i)}.
        \end{equation}
    \end{enumerate}
    Moreover, $(X,R)$ is \emph{commutative} [resp.\ \emph{symmetric}] if \ref{enu: AS commutativity} [resp.\ \ref{enu: AS symmetry}] holds.
    \begin{enumerate}[resume* = AS]
        \item \label{enu: AS commutativity} $p_{i,i'}^k = p_{i',i}^k$ for all $i,i',k \in \mathcal{I}$.
        \item \label{enu: AS symmetry} $i^\top = i$ for all $i \in \mathcal{I}$.
    \end{enumerate}
\end{definition}

This structure provides a purely combinatorial framework in which harmonic analysis can be formulated in algebraic terms.
Association schemes have since become a central object of study in algebraic combinatorics \cite{bannai2021algebraic, MR384310}.

On the other hand, on compact homogeneous spaces, harmonic analysis has been established (cf. \cite{MR2328043}), and combinatorial aspects such as designs and codes have been studied in \cite{MR485471, MR796472, MR2529619, MR1941981, MR2033722, MR679208, MR4097913, MR4331068, MR4780765}.
A natural question arising in this context is whether one can define a ``compact version'' of association schemes that encompasses both classical association schemes and compact homogeneous spaces, and provides a framework for developing harmonic analysis and combinatorial theory on them.

Several generalizations of association schemes that drop the finiteness assumption have been proposed, including the following examples:
countable discrete association schemes \cite{MR2184345};
totally disconnected association schemes \cite{MR3366838};
profinite association schemes \cite{MatsumotoOgawaOkuda2022};
Delsarte spaces, which are association schemes on metric spaces equipped with measures \cite{combcon};
and continuous association schemes \cite{Voit2019}.
Among these, the frameworks capable of treating rank-one compact symmetric spaces are Delsarte spaces and continuous association schemes.
In particular, only continuous association schemes are general enough to handle arbitrary compact homogeneous spaces.
However, the definition of continuous association schemes is formulated in terms of hypergroups, and its relationship to the classical notion of association schemes is not immediately transparent.

In this paper, we propose a definition of \emph{compact association schemes} from a point of view different from Voit’s.
Our definition is expressed in a form closely analogous to that of finite association schemes, admits a naturally defined Bose--Mesner algebra, and includes as a subclass the compact \emph{strong} continuous association schemes in the sense of \cite{Voit2019}.

We remark that harmonic analysis can be developed on the compact association schemes introduced here; this aspect will be discussed in a separate work.

\section{Compact association schemes}

In this section, we introduce a definition of compact association schemes.
Let $X$ be a non-empty compact Hausdorff space and let $\mu_X$ be a strictly positive Radon measure on $X$.
Note that $\mu_X$ is a regular finite Borel measure.

\begin{definition}
    Let $\mathcal{I}$ be a compact Hausdorff space and $R \colon X \times X \to \mathcal{I}$ be a quotient map.
    We call $R \coloneq (X,R)$ a \emph{compact association scheme} if it satisfies the following conditions:
    \begin{enumerate}[({CAS}1), series = CAS]
        \item \label{enu: CAS i_0} There exists $i_0 \in \mathcal{I}$ such that $R^{-1}(i_0) = \Delta_X \coloneq \set{(x,x) | x \in X}$.
        \item \label{enu: CAS intersection number} For all Borel sets $W, W' \subset \mathcal{I}$ and $k \in \mathcal{I}$, there exists $p_{W, W'}^k \in \R_{\ge 0}$ such that for all $(x,z) \in R^{-1}(k)$,
        \begin{equation}
            p_{W, W'}^k
            = \mu_X \ab(\set{y \in X | R(x,y) \in W,\ R(y,z) \in W'}).
        \end{equation}
        \item \label{enu: CAS transpose} For all $i \in \mathcal{I}$, there exists $i^\top \in \mathcal{I}$ such that
        \begin{equation}
            R^{-1}(i^\top)
            = R^{-1}(i)^\top
            \coloneq \set{(y,x) | (x,y) \in R^{-1}(i)}.
        \end{equation}
    \end{enumerate}
    Moreover, $(X,R)$ is \emph{commutative} [resp.\ \emph{symmetric}] if \ref{enu: CAS commutativity} [resp.\ \ref{enu: CAS symmetry}] holds.
    \begin{enumerate}[resume* = CAS]
        \item \label{enu: CAS commutativity} $p_{W, W'}^k = p_{W', W}^k$ for all Borel sets $W, W' \subset \mathcal{I}$ and $k \in \mathcal{I}$.
        \item \label{enu: CAS symmetry} $i^\top = i$ for all $i \in \mathcal{I}$.
    \end{enumerate}
\end{definition}

We can see from $p_{W, W'}^k = p_{W'^\top, W^\top}^{k^\top}$ that any symmetric compact association scheme is commutative.

We now present several examples of compact association schemes.
It is convenient to use \zcref{p: useful observation} (\ref{enu: easy to check CAS}) when verifying that each example indeed defines a compact association scheme.

\begin{example}[Association scheme]
    Any (finite) association scheme is a compact association scheme.
    Conversely, an association scheme is a compact association scheme $(X,R)$ such that $X$ is a finite discrete space and $\mu_X$ is the counting measure.
    Commutativity and symmetry are also defined consistently.
\end{example}

\begin{example}[Compact homogeneous space]
    Let $G$ be a compact Hausdorff topological group, $X$ be a non-empty compact Hausdorff homogeneous $G$-space and $\mu_X$ be the $G$-invariant probability Radon measure on $X$.
    Then the quotient $R \colon X \times X \to X \times X / \operatorname{diag} G$ of the diagonal $G$-action is a compact association scheme.
    Moreover, commutativity of compact homogeneous spaces, regarded as compact association schemes, is equivalent to the multiplicity-freeness of their regular representations.
\end{example}

\begin{example}[Delsarte space]\label{ex: Delsarte space}
    Let $(X,d)$ be a compact metric space equipped with a strictly positive Radon measure $\mu_X$, and assume that $X$ is a Delsarte space with respect to $\mu_X$ (cf. \cite[p.\ 52]{combcon}).
    Then the squared distance function $R \coloneq d^2 \colon X \times X \to \mathcal{I} \coloneq R(X \times X)$ is a symmetric compact association scheme.
\end{example}

Next, we define Bose--Mesner algebras corresponding to compact association schemes.
We begin by briefly recalling the notion of nets.
A non-empty set $(\mathcal{N}, \le)$ endowed with a binary relation that is reflexive and transitive, and such that any two elements admit a common upper bound, is called a directed set.
A map $a \colon \mathcal{N} \to Y$ from a directed set to a set is called a net, written $(a_N)_{N \in \mathcal{N}} \subset Y$.
When $Y$ is a normed space, the net $(a_N)_{N \in \mathcal{N}}$ is said to converge to $a \in Y$ if, for all $\varepsilon > 0$, there exists $N_0 \in \mathcal{N}$ such that $\norm{a - a_N} < \varepsilon$ for all $N \ge N_0$.
In this case we write $\lim_N a_N = a$.

We next describe several operations on the space $\Mat(X) \coloneq C(X \times X)$ of all $\C$-valued continuous functions on $X \times X$.
The space $\Mat(X)$ is equipped with matrix multiplication $\ast$, Hadamard multiplication $\circ$, transpose $-^\top$ and complex conjugate $\overline{-}$, where matrix multiplication $\ast$ is defined by
\begin{equation}
    A \ast B(x,z)
    \coloneq \int_y A(x,y) \cdot B(y,z) \d \mu_X
    \quad (A,B \in \Mat(X),\ x,z \in X).
\end{equation}
These operations are continuous with respect to the uniform norm $\norm{-}_\infty$.

We are now ready to define the Bose--Mesner algebra on $X$:

\begin{definition}
    Let $\mathfrak{A}$ be a unital C*-subalgebra of $(\Mat(X), \circ, \overline{-}, \norm{-}_\infty)$.
    We call $\mathfrak{A}$ a \emph{Bose--Mesner algebra} if it satisfies the following conditions:
    \begin{enumerate}[({BMA}1), series = BMA]
        \item \label{enu: BMA identity}
        \begin{enumerate}[ref=\theenumi(\alph*)]
            \item \label{enu: BMA I} There exists a net $(I_N)_{N \in \mathcal{N}}$ in $\mathfrak{A}$ such that for all $A \in \Mat(X)$, we have $\lim_N I_N \ast A = A = \lim_N A \ast I_N$, i.e., $(I_N)_{N \in \mathcal{N}}$ is an approximate identity of the Banach algebra $(\Mat(X), \ast, \norm{-}_\infty)$.
            \item \label{enu: BMA J} $\mathfrak{A} \ast J \subset \C J$, where $J$ is the constant function with value $1$.
        \end{enumerate}
        \item \label{enu: BMA matrix multiplication} $\mathfrak{A}$ is closed under matrix multiplication.
        \item \label{enu: BMA transpose} $\mathfrak{A}$ is closed under transpose.
    \end{enumerate}
    Moreover, $\mathfrak{A}$ is \emph{commutative} [resp.\ \emph{symmetric}] if \ref{enu: BMA commutativity} [resp.\ \ref{enu: BMA symmetry}] holds.
    \begin{enumerate}[resume* = BMA]
        \item \label{enu: BMA commutativity} Matrix multiplication on $\mathfrak{A}$ is commutative.
        \item \label{enu: BMA symmetry} $A^\top = A$ holds for all $A \in \mathfrak{A}$.
    \end{enumerate}
\end{definition}

Now, Gelfand duality, that is a duality of the category of compact Hausdorff spaces and the category of unital commutative C*-algebras, induces a correspondence between quotient maps $R$ of $X \times X$ and unital C*-subalgebras $\mathfrak{A}$ of $\Mat(X)$ (cf. \cite{MR1074574}).
Moreover, we prove that $R$ being a compact association scheme is equivalent to $\mathfrak{A}$ being a Bose--Mesner algebra:

\begin{theorem} \label{p: CH to C*-Alg}
    Let $\mathcal{I}$ be a compact Hausdorff space and $R \colon X \times X \to \mathcal{I}$ be a quotient map.
    Then the corresponding unital C*-subalgebra
    \begin{equation}
        \mathfrak{A}_R
        \coloneq \set{R^* f \coloneq f \circ R | f \in C(\mathcal{I})}
    \end{equation}
    of $\Mat(X)$ is a Bose--Mesner algebra if and only if $R$ is a compact association scheme.
    Moreover, $\mathfrak{A}_R$ is commutative [resp.\ symmetric] precisely when $R$ is commutative [resp.\ symmetric].
\end{theorem}

\begin{theorem} \label{p: C*-Alg to CH}
    Let $\mathfrak{A}$ be a unital C*-subalgebra of $\Mat(X)$.
    Then the corresponding quotient map
    \begin{align}
        R_\mathfrak{A} \colon X \times X &\to \mathcal{I}_\mathfrak{A} \coloneq \set{i \colon \mathfrak{A} \to \C | \text{$i$ is a $\circ$-algebraic homomorphism}} \\
        (x,y) &\mapsto (A \mapsto A(x,y))
    \end{align}
    is a compact association scheme if and only if $\mathfrak{A}$ is a Bose--Mesner algebra.
    Moreover, $R_\mathfrak{A}$ is commutative [resp.\ symmetric] precisely when $\mathfrak{A}$ is commutative [resp.\ symmetric].
\end{theorem}

These results are analogous to those in the finite case of association schemes.
To prove these two theorems, we will prove two lemmas.

\begin{lemma} \label{p: CAS to BMA}
    Let $R \colon X \times X \to \mathcal{I}$ be a compact association scheme.
    Then $\mathfrak{A}_R$ is a Bose--Mesner algebra.
    Moreover, if $R$ is commutative [resp.\ symmetric], then $\mathfrak{A}_R$ is also commutative [resp.\ symmetric].
\end{lemma}

\begin{proof}
    \ref{enu: BMA transpose}:
    It is clear from \ref{enu: CAS transpose}.

    \ref{enu: BMA matrix multiplication}:
    From \ref{enu: CAS intersection number},
    \begin{equation}
        \int_y R^* \chi_W(x,y) \cdot R^* \chi_{W'}(y,z) \d \mu_X
        = p_{W,W'}^{R(x,z)}
    \end{equation}
    holds for all Borel sets $W,W' \subset \mathcal{I}$ and $x,z \in X$, where $\chi_W$ denotes the characteristic function of $W$.
    Hence, $R^* f \ast R^* g(x,z)$ is a constant independent of $(x,z) \in R^{-1}(k)$ for all $f,g \in C(\mathcal{I})$ and $k \in \mathcal{I}$, because it can be written as
    \begin{equation}
        R^* f \ast R^* g(x,z)
        = \lim_{n \to \infty} \int_y R^* f_n(x,y) \cdot R^* g_n(y,z) \d \mu_X,
    \end{equation}
    considering uniform approximations $f_n \to f,\ g_n \to g \ (n \to \infty)$ by simple functions.
    Thus, $\mathfrak{A}_R$ is closed under matrix multiplication.

    \ref{enu: BMA J}:
    From \ref{enu: BMA matrix multiplication} and \ref{enu: CAS i_0}, $A \ast J$ is a constant function for all $A \in \mathfrak{A}_R$.

    \ref{enu: BMA I}:
    We consider the neighbourhood system $\mathcal{N}(i_0)$ of $i_0 \in \mathcal{I}$.
    Since $\mathcal{I}$ is compact Hausdorff,
    for each $N \in \mathcal{N}(i_0)$, there exists a continuous function $h_N \colon \mathcal{I} \to [0,1]$ such that $h_N(i_0) = 1,\ h_N|_{N^c} = 0$.
    Then we put
    \begin{equation}
        I_N
        \coloneq \frac{\mu_X(X)}{2 \int R^* h_N \d \mu_X \otimes \mu_X} \cdot (R^* h_N + R^* h_N^\top),
    \end{equation}
    where $\mu_X \otimes \mu_X$ denotes the product measure.
    Note that $I_N$ is in $\mathfrak{A}_R$ by \ref{enu: BMA transpose}, and $\int_y I_N(x,y) \d \mu_X = 1$ holds for all $x \in X$ by \ref{enu: BMA J}.

    Below, we show that the net $(I_N)_{N \in (\mathcal{N}(i_0), \supset)} \subset \mathfrak{A}_R$ is an approximate identity.
    Fix $A \in \Mat(X)$ and $\varepsilon > 0$.
    Because $\set{A(-,z) | z \in X} \subset C(X)$ is compact, by the Arzel\`a--Ascoli theorem, $\set{A(-,z) | z \in X}$ is uniformly equicontinuous, that is, there exists a neighborhood $U$ of $\Delta_X$ such that
    \begin{equation}
        \sup_{(x,y) \in U,\ z \in X} \abs{A(x,z) - A(y,z)}
        < \varepsilon
    \end{equation}
    holds.
    Now, we have $R(U^c)^c \cap R(U^{c \top})^c \in \mathcal{N}(i_0)$ since $R$ is closed and \ref{enu: CAS i_0} holds, and for all $N \in \mathcal{N}(i_0)$ with $N \subset R(U^c)^c \cap R(U^{c \top})^c$ and all $x,z \in X$,
    \begin{equation}
        \abs{(A - I_N \ast A)(x,z)}
        \le \int_y I_N(x,y) \abs{A(x,z) - A(y,z)} \d \mu_X
        < \varepsilon.
    \end{equation}
    Thus, $I_N \ast A$ converges uniformly to $A$, and so does $A \ast I_N$.
    Therefore, $\mathfrak{A}_R$ is a Bose--Mesner algebra.

    \ref{enu: BMA commutativity}:
    It can be shown by the same approach as \ref{enu: BMA matrix multiplication} from \ref{enu: CAS commutativity}.

    \ref{enu: BMA symmetry}:
    It is easy to see from \ref{enu: CAS symmetry}.
    \qed
\end{proof}

\begin{lemma}
    Let $\mathfrak{A} \subset \Mat(X)$ be a Bose--Mesner algebra.
    Then $R_\mathfrak{A}$ is a compact association scheme.
    Moreover, if $\mathfrak{A}$ is commutative [resp.\ symmetric], then $R_\mathfrak{A}$ is also commutative [resp.\ symmetric].
\end{lemma}

\begin{proof}
    \ref{enu: CAS transpose}:
    \ref{enu: BMA transpose} ensures that $i^\top \colon \mathfrak{A} \to \C;\ A \mapsto i(A^\top)$ is well-defined for every $i \in \mathcal{I}_\mathfrak{A}$.
    It is easy to check that $i^\top$ satisfies the required conditions.

    \ref{enu: CAS i_0}:
    Fix $x_0 \in X$ and put $i_0 \coloneq R_\mathfrak{A}(x_0,x_0) \in \mathcal{I}_\mathfrak{A}$.
    Let $(I_N)_{N \in \mathcal{N}} \subset \mathfrak{A}$ be an approximate identity obtained from \ref{enu: BMA I}.
    Then for all $A \in \mathfrak{A}$, we have
    \begin{equation}
        A(x_0,x_0)
        = \lim_N A \ast I_N(x_0,x_0)
        = \lim_N ((A \circ I_N^\top) \ast J(x_0,x_0)).
    \end{equation}
    Hence, by \ref{enu: BMA J} and \ref{enu: BMA transpose}, $i_0$ is independent of the choice of $x_0$.

    Next, take $x \in X$ with $R_\mathfrak{A}(x_0,x) = i_0$.
    For all $N \in \mathcal{N}$,
    \begin{align}
        \MoveEqLeft
        \int_y \abs{I_N(x_0,y) - I_N(x,y)}^2 \d \mu_X
        \\&= I_N \ast \overline{I_N^\top}(x_0,x_0) + I_N \ast \overline{I_N^\top}(x,x) - I_N \ast \overline{I_N^\top}(x_0,x) - \overline{I_N} \ast I_N^\top(x_0,x)
        \\&= 0
    \end{align}
    follows from \ref{enu: BMA matrix multiplication} and \ref{enu: BMA transpose}; thus, we have $I_N(x_0,-) = I_N(x,-)$.
    Since $(I_N)_{N \in \mathcal{N}}$ is an approximate identity, we have $A(x_0,x_0) = A(x,x_0)$ for all $A \in \Mat(X)$, and therefore we get $x_0 = x$.
    Due to the above, $R_\mathfrak{A}^{-1}(i_0) = \Delta_X$ is satisfied.

    \ref{enu: CAS intersection number}:
    We consider the pushforward measures $R_{\mathfrak{A} *} \mu_X \otimes \mu_X$ and $R_{\mathfrak{A} x *} \mu_X$, where $R_{\mathfrak{A} x}$ denotes the map $R_\mathfrak{A}$ with $x \in X$ fixed.
    Because \ref{enu: BMA J} and $\mathfrak{A}_{R_\mathfrak{A}} = \mathfrak{A}$ holds, we have $R_{\mathfrak{A} *} \mu_X \otimes \mu_X = \mu_X(X) R_{\mathfrak{A} x *} \mu_X$.
    Then fix Borel sets $W_1, W_2 \subset \mathcal{I}_\mathfrak{A}$ and $k \in \mathcal{I}_\mathfrak{A}$, and take $L^2$-approximations $f_n \to \chi_{W_1},\ g_n \to \chi_{W_2} \ (n \to \infty)$ of the characteristic functions $\chi_{W_1},\chi_{W_2}$ by continuous functions with respect to $R_{\mathfrak{A} *} \mu_X \otimes \mu_X$.
    From the Cauchy--Schwarz inequality, we have
    \begin{align}
        \MoveEqLeft
        \norm{R_\mathfrak{A}^* \chi_{W_1} \ast R_\mathfrak{A}^* \chi_{W_2} - R_\mathfrak{A}^* f_n \ast R_\mathfrak{A}^* g_n}_\infty
        \\&\le \sup_{x,z \in X} \abs{\int_y R_\mathfrak{A}^*(\chi_{W_1} - f_n)(x,y) \cdot R_\mathfrak{A}^* \chi_{W_2}(y,z) \d \mu_X}
        \\&\phantom{\le \quad} + \sup_{x,z \in X} \abs{\int_y R_\mathfrak{A}^* f_n(x,y) \cdot R_\mathfrak{A}^*(\chi_{W_2} - g_n)(y,z) \d \mu_X}
        \\&\le \sup_{x,z \in X} \ab(\int_y \abs{R_\mathfrak{A}^*(\chi_{W_1} - f_n)(x,y)}^2 \d \mu_X \cdot \int_y \abs{R_\mathfrak{A}^* \chi_{W_2}(y,z)}^2 \d \mu_X)^{1/2}
        \\&\phantom{\le \quad} + \sup_{x,z \in X} \ab(\int_y \abs{R_\mathfrak{A}^* f_n(x,y)}^2 \d \mu_X \cdot \int_y \abs{R_\mathfrak{A}^*(\chi_{W_2} - g_n)(y,z)}^2 \d \mu_X)^{1/2}
        \\&= \frac{1}{\mu_X(X)} \ab(\int \abs{\chi_{W_1} - f_n}^2 \d R_{\mathfrak{A} *} \mu_X \otimes \mu_X \cdot \int \abs{\chi_{W_2}}^2 \d R_{\mathfrak{A} *} \mu_X \otimes \mu_X)^{1/2}
        \\&\phantom{\le \quad} + \frac{1}{\mu_X(X)} \ab(\int \abs{f_n}^2 \d R_{\mathfrak{A} *} \mu_X \otimes \mu_X \cdot \int \abs{\chi_{W_2} - g_n}^2 \d R_{\mathfrak{A} *} \mu_X \otimes \mu_X)^{1/2}
        \\&\to 0 \quad (n \to \infty).
    \end{align}
    Since $\mathfrak{A}_{R_\mathfrak{A}}$ is complete, $R_\mathfrak{A}^* \chi_{W_1} \ast R_\mathfrak{A}^* \chi_{W_2} \in \mathfrak{A}_{R_\mathfrak{A}} = \mathfrak{A}$ holds.
    Thus
    \begin{equation}
        \mu_X(\set{y \in X | R_\mathfrak{A}(x,y) \in W_1,\ R_\mathfrak{A}(y,z) \in W_2})
        = R_\mathfrak{A}^* \chi_{W_1} \ast R_\mathfrak{A}^* \chi_{W_2} (x,z)
    \end{equation}
    is a constant independent of $(x,z) \in R_\mathfrak{A}^{-1}(k)$.
    Therefore, $R_\mathfrak{A}$ is a compact association scheme.

    \ref{enu: CAS commutativity}:
    It can be shown by the same approach as \ref{enu: CAS intersection number} from \ref{enu: BMA commutativity}.

    \ref{enu: CAS symmetry}:
    For all $i \in \mathcal{I}_\mathfrak{A}$ and $A \in \mathfrak{A}$, we have $i^\top(A) = i(A^\top) = i(A)$ from \ref{enu: BMA symmetry}.
    \qed
\end{proof}

By the above, the theorems are proven:
\begin{proof}[\zcref{p: CH to C*-Alg} and \zcref{p: C*-Alg to CH}]
    Since $R_{\mathfrak{A}_R}$ is equivalent to $R$ for all quotient map $R$ of $X \times X$, \zcref{p: CH to C*-Alg} is shown.
    Also, since $\mathfrak{A}_{R_\mathfrak{A}}$ is equal to $\mathfrak{A}$ for all unital C*-subalgebra $\mathfrak{A}$ of $\Mat(X)$, \zcref{p: C*-Alg to CH} is shown.
    \qed
\end{proof}

Furthermore, we obtain the following useful observations, which follow from the proofs of the theorems:

\begin{proposition}\label{p: useful observation}
    The following holds:
    \begin{enumerate}
        \item Any Bose--Mesner algebra $\mathfrak{A}$ admits a symmetric and non-negative approximate identity.
        \item For any compact association scheme $R$, the equality $R_* \mu_X \otimes \mu_X = \mu_X(X) R_{x *} \mu_X$ holds for all $x \in X$.
        \item\label{enu: easy to check CAS} To show that a quotient map $R$ of $X \times X$ is a compact association scheme, it suffices to verify that $R$ satisfies \ref{enu: CAS i_0} and that the associated algebra $\mathfrak{A}_R$ satisfies \ref{enu: BMA matrix multiplication} and \ref{enu: BMA transpose}.
    \end{enumerate}
\end{proposition}

\section{Compact strong continuous association schemes are compact association schemes}

In this section, we focus on compact strong continuous association schemes introduced by Voit \cite{Voit2019}.
Here, we do not restate the definitions; instead, we collect several facts that follow from them and from the literature, which will be used later on.

\begin{fact} \label{p: fact of VAS}
    Let $(X, \mathcal{I}, \kappa)$ be a compact strong continuous association scheme.
    Then the following hold:
    \begin{enumerate}[({VAS}1)]
        \item $X$ and $\mathcal{I}$ are second-countable compact Hausdorff spaces \cite[Def. 4.2]{Voit2019}.

        \item $\kappa$ is a map from $X \times \mathcal{I}$ to the space $\mathrm{Prob}(X)$ of all probability measures on $X$ with $((x,i) \mapsto \int \varphi \d \kappa(x,i)) \in C(X \times \mathcal{I})$ for all $\varphi \in C(X)$, that is, $\kappa$ is a continuous Markov kernel from $X \times \mathcal{I}$ to $X$ \cite[Def. 4.2]{Voit2019}.

        \item \label{enu: VAS R} The map
        \begin{align}
            R \colon X \times X &\to \mathcal{I} \\
            (x,y) &\mapsto i \text{ if } y \in \supp \kappa(x,i)
        \end{align}
        is a well-defined quotient map \cite[Def. 4.2]{Voit2019}.

        \item \label{enu: VAS i_0} There exists $i_0 \in \mathcal{I}$ such that $\kappa(x,i_0) = \delta_x$ for all $x \in X$, where $\delta_x$ is a Dirac measure on $x$ \cite[Def. 4.2]{Voit2019}.

        \item There exist a strictly positive Radon measure $\mu_X$ on $X$, called the \emph{invariant measure}; a binary operation $*$ on $\mathrm{Prob}(\mathcal{I})$, called the \emph{hypergroup convolution}; and a continuous involution $-^\top \colon \mathcal{I} \to \mathcal{I}$, called the \emph{hypergroup involution} \cite[Def. 4.2]{Voit2019}.

        \item \label{enu: HG convolution of functions} For all $f,g \in C(\mathcal{I})$, the convolution product
        \begin{equation}
            f * g(i)
            \coloneq \int_{i'} \int f \d \delta_i * \delta_{i'} \cdot g(i'^\top) \d \mu_\mathcal{I}
            \quad (i \in \mathcal{I})
        \end{equation}
        satisfies $f * g \in C(\mathcal{I})$ \cite[Prop. 1.4.9]{MR1312826}.

        \item \label{enu: VAS T1} For all $i \in \mathcal{I},\ x,z \in X$ and $f \in C(\mathcal{I})$, we have
        \begin{equation}
            \int f \d \delta_{R(x,z)} * \delta_i
            = \int_y R^* f(x,y) \d \kappa(z,i).
        \end{equation}
        See \cite[Def. 5.1]{Voit2019}.

        \item \label{enu: VAS T2} There exists a left Haar measure $\mu_\mathcal{I}$ on $\mathcal{I}$ such that
        \begin{equation}
            \int_i f(i) \cdot \int \varphi \d \kappa(x,i) \d \mu_\mathcal{I}
            = \int_y R^* f(x,y) \cdot \varphi(y) \d \mu_X
        \end{equation}
        holds for all $x \in X,\ f \in C(\mathcal{I})$ and $\varphi \in C(X)$ \cite[Def. 5.1]{Voit2019}.

        \item \label{enu: VAS transpose} $R(x,y)^\top = R(y,x)$ holds for all $x,y \in X$ \cite[Fact 4.7]{Voit2019}.
        Moreover, the \emph{transpose} $f^\top$ of a function $f$ on $\mathcal{I}$ is defined as the pullback of $f$ by $-^\top$, and the \emph{transpose} $\mu^\top$ of a measure $\mu$ on $\mathcal{I}$ is defined as the pushforward of $\mu$ by $-^\top$.

        \item \label{enu: mu_I is invariant under involution} $\mu_\mathcal{I}^\top = \mu_\mathcal{I}$ holds \cite[pp. 28, 40]{MR1312826}.

        \item \label{enu: HG anti-automorphism} $(\delta_i * \delta_{i'})^\top = \delta_{i'^\top} * \delta_{i^\top}$ holds for all $i,i' \in \mathcal{I}$ \cite[Def. 2.1, 4.2]{Voit2019}.

        \item \label{enu: VAS commutativity} We say that $(X, \mathcal{I}, \kappa)$ is \emph{commutative} if the hypergroup convolution $*$ on $\mathrm{Prob}(\mathcal{I})$ is commutative \cite[Def. 2.1, 4.2]{Voit2019}.

        \item \label{enu: VAS symmetry} We say that $(X, \mathcal{I}, \kappa)$ is \emph{symmetric} if $i^\top = i$ for all $i \in \mathcal{I}$ \cite[Def. 2.1, 4.2]{Voit2019}.
    \end{enumerate}
\end{fact}

Now we prove that any compact strong continuous association scheme is a compact association scheme:

\begin{theorem}
    Let $(X, \mathcal{I}, \kappa)$ be a compact strong continuous association scheme.
    Then the quotient map $R$ obtained from \ref{enu: VAS R} is a compact association scheme.
    Moreover, if $(X, \mathcal{I}, \kappa)$ is commutative [resp.\ symmetric], then $R$ is also commutative [resp.\ symmetric].
\end{theorem}

\begin{proof}
    We show that $R$ satisfies \ref{enu: CAS i_0} and that the associated algebra $\mathfrak{A}_R$ satisfies \ref{enu: BMA matrix multiplication} and \ref{enu: BMA transpose}.

    \ref{enu: CAS i_0}:
    It follows from \ref{enu: VAS i_0}.

    \ref{enu: BMA matrix multiplication}:
    Fix $f,g \in C(\mathcal{I})$.
    From \ref{enu: VAS T1}--\ref{enu: VAS transpose}, for all $x,z \in X$, we have
    \begin{align}
        R^* f \ast R^* g(x,z)
        &= \int_y R^* f(x,y) \cdot R^* g^\top(z,y) \d \mu_X
        \\&= \int_i \int_y R^* f(x,y) \d \kappa(z,i) \cdot g^\top(i) \d \mu_\mathcal{I}
        \\&= \int_i \int f \d \delta_{R(x,z)} * \delta_i \cdot g^\top(i) \d \mu_\mathcal{I}
        \\&= R^*(f * g)(x,z).
    \end{align}
    Thus, by \ref{enu: HG convolution of functions}, we have $R^* f \ast R^* g = R^*(f * g) \in \mathfrak{A}_R$.

    \ref{enu: BMA transpose}:
    It is clear from \ref{enu: VAS transpose}.

    Hence, by \zcref{p: useful observation} (\ref{enu: easy to check CAS}), $R$ is a compact association scheme.

    \ref{enu: BMA commutativity}:
    For all $f,g \in C(\mathcal{I})$ and $x,z \in X$, from \ref{enu: mu_I is invariant under involution}--\ref{enu: VAS commutativity}, we get
    \begin{align}
        R^* f \ast R^* g (x,z)
        &= R^* g^\top \ast R^* f^\top (z,x)
        \\&= \int_i \int g^\top \d \delta_{R(z,x)} * \delta_i \cdot f(i) \d \mu_\mathcal{I}
        \\&= \int_i \int g \d \delta_i * \delta_{R(x,z)} \cdot f^\top(i) \d \mu_\mathcal{I}
        \\&= \int_i \int g \d \delta_{R(x,z)} * \delta_i \cdot f^\top(i) \d \mu_\mathcal{I}
        \\&= R^* g \ast R^* f(x,z).
    \end{align}

    \ref{enu: BMA symmetry}:
    It is trivial from \ref{enu: VAS symmetry}.
    \qed
\end{proof}

\section{Acknowledgments}

The author would like to express his sincere gratitude to Takayuki Okuda and Kento Ogawa, as well as to Eiichi Bannai, Hirotake Kurihara, Makoto Matsumoto, Takuya Saito, and Takashi Satomi, for their many valuable discussions and helpful comments.

\end{document}